\documentclass[reqno]{amsart}
\usepackage{amsthm,amsmath,amssymb}
\usepackage{stmaryrd,mathrsfs}
\newcommand{\ud}[0]{\,\mathrm{d}}
\newcommand{\eps}[0]{\varepsilon}

\newcommand{\bddlin}[0]{\mathscr{L}}
\newcommand{\rbound}[0]{\mathscr{R}}
\newcommand{\floor}[1]{\lfloor #1\rfloor}
\newcommand{\ceil}[1]{\lceil #1\rceil}

\newcommand{\abs}[1]{|#1|}
\newcommand{\Babs}[1]{\Big|#1\Big|}

\newcommand{\Norm}[2]{\|#1\|_{#2}}

\newcommand{\BNorm}[2]{\Big\|#1\Big\|_{#2}}

\newcommand{\prob}[0]{\mathbb{P}}

\newcommand{\R}{\mathbb{R}}
\newcommand{\C}{\mathbb{C}}
\newcommand{\N}{\mathbb{N}}
\newcommand{\Z}{\mathbb{Z}}

\numberwithin{equation}{section}

\makeatletter
  \let\c@equation\c@subsection
\makeatother

\theoremstyle{plain}
\newtheorem{theorem}[subsection]{Theorem}

\newtheorem{corollary}[subsection]{Corollary}
\newtheorem{lemma}[subsection]{Lemma}

\theoremstyle{definition}

\theoremstyle{remark}
\newtheorem{remark}[subsection]{Remark}

\begin{document}

\title[The vector-valued multiplier theorem]{New thoughts on the vector-valued Mihlin--H\"ormander multiplier theorem}
%{The fall of Fourier-type in vector-valued multiplier theory}

\author[T.~P.\ Hyt\"onen]{Tuomas P.\ Hyt\"onen}
\address{Department of Mathematics and Statistics, University of Helsinki, P.O.B. 68, FI-00014 Helsinki, Finland}
\email{tuomas.hytonen@helsinki.fi} %Gustaf H\"all\-str\"omin katu 2b

\date{\today}

\subjclass[2000]{42B15 (Primary); 46B09, 46B20 (Secondary)}
\keywords{Fourier multiplier, type and cotype of Banach spaces}

% 42B15 Harmonic analysis in several variables: Multipliers
% 42B20 Harmonic analysis in several variables: Singular integrals (Calder\'on-Zygmund, etc.)
% 42B25 Harmonic analysis in several variables: Maximal functions, Littlewood-Paley theory
% 46B09 Probabilistic methods in Banach space theory
% 46B20 Geometry and structure of normed linear spaces
% 46E40 Spaces of vector- and operator-valued functions
% 60G46 Martingales and classical analysis

\begin{abstract}
Let $X$ be a UMD space with type $t$ and cotype $q$, and let $T_m$ be a Fourier multiplier operator with a scalar-valued symbol $m$. If $\abs{\partial^{\alpha}m(\xi)}\lesssim\abs{\xi}^{-\abs{\alpha}}$ for all $\abs{\alpha}\leq\floor{n/\max(t,q')}+1$, then $T_m$ is bounded on $L^p(\R^n;X)$ for all $p\in(1,\infty)$. For scalar-valued multipliers, this improves the theorem of Girardi and Weis (J. Funct. Anal., 2003) who required similar assumptions for derivatives up to the order $\floor{n/r}+1$, where $r\leq\min(t,q')$ is a Fourier-type of $X$. However, the present method does not apply to operator-valued multipliers, which are also covered by the Girardi--Weis theorem.
\end{abstract}

\maketitle

%\tableofcontents

\section{Introduction}\label{intro}

For a function $m$ defined on $\R^n\setminus\{0\}$, the classical Mihlin multiplier condition is $\abs{\partial^{\alpha}m(\xi)}\lesssim\abs{\xi}^{-\abs{\alpha}}$ for all $\alpha\in\{0,1\}^n$. By Mihlin's theorem \cite{Mihlin}, this is sufficient for the $L^p(\R^n)$ boundedness of the multiplier operator
\begin{equation*}
  T_m f(x)=\int_{\R^n}m(\xi)\hat{f}(\xi)e^{i2\pi\xi\cdot x}\ud\xi,
\end{equation*}
where $\hat{f}$ is the Fourier transform of $f$. This can also be written as $T_m f=K*f$, where $K=\check{m}$ is the inverse Fourier transform of $m$ in the sense of distributions. Exploiting the convolution point-of-view, H\"ormander proved a variant of the multiplier theorem \cite{Hormander}, where Mihlin's condition is assumed for the derivatives corresponding to the multi-indices of length $\abs{\alpha}\leq\floor{n/2}+1$, with $\floor{x}:=\max\{k\in\Z:k\leq x\}$. As a matter of fact, one may intersect the assumptions of these two theorems, assuming the bounds only for the multi-indices satisfying both $\alpha\in\{0,1\}^n$ and $\abs{\alpha}\leq\floor{n/2}+1$ (see \cite{Hyt:Mihlin}).

%Instead of full derivatives, one may consider H\"older continuity assumptions of the type
%\begin{equation*}
%  \Babs{\prod_{i:\alpha_i=1}[I-\tau_{h_i e_i}]m(\xi)}\lesssim\abs{h^{\alpha\gamma}}\cdot\abs{\xi}^{-\abs{\alpha}\gamma},\qquad
%  \abs{h}<\frac12\abs{\xi},
%\end{equation*}
%where $\tau_h m(\xi):=m(\xi-h)$ and $h^{\alpha}:=\prod_{i=1}^n h_i^{\alpha_i}$. With $\gamma>1/2$, this is still a sufficient condition for the $L^p(\R^n)$ boundedness of the multiplier operator
%\begin{equation*}
%  T_m f(x)=\int_{\R^n}m(\xi)\hat{f}(\xi)e^{i2\pi\xi\cdot x}\ud\xi,
%\end{equation*}
%which can also be written as $T_m f=K*f$, where $K=\check{m}$ is the distributional inverse Fourier transform of $m$. This is due to H\"ormander. A sufficient condition for the $(1/2+\eps)$-H\"older continuity in terms of full derivatives is to check Mihlin's condition for all $\abs{\alpha}\leq\lfloor n/2\rfloor +1$, which is a more usual formulation of H\"ormander's multiplier condition.
%
Bourgain \cite{Bourgain}, McConnell \cite{McConnell} and Zimmermann \cite{Zimmermann} extended Mihlin's (but not H\"ormander's) multiplier theorem to the case of $L^p(\R^n;X)$, the Bochner space with values in a UMD space $X$ (for the definition of UMD, see e.g. \cite{Bourgain}). McConnell also observed that H\"ormander's condition suffices for special multipliers, which are supported in a dyadic annulus $r<\abs{\xi}<2r$, a result which will be recovered here by a different method.

A perhaps more satisfactory vector-valued version of H\"ormander's multiplier theorem was obtained by Girardi and Weis \cite{GW}. They showed that Mihlin bounds for derivatives up to the order $\floor{n/r}+1$ suffice if the UMD space $X$ also has {Fourier-type} $r\in(1,2]$, i.e., the Hausdorff--Young inequality holds for the vector-valued Fourier transform $\mathscr{F}:L^r(\R^n;X)\to L^{r'}(\R^n;X)$. Their result even applies to operator-valued multipliers, with an appropriate formulation of the assumptions.

However, the aim of this note is to show that if one is just interested in scalar-valued multipliers, the Fourier-type assumption may be substantially relaxed. I will show that instead of derivatives up to the order $\floor{n/r}+1$, with $r$ a Fourier-type, it suffices to go only up to the order $\floor{n/\max(t,q')}+1$, where $t$ is a {type} and $q$ a {cotype} for $X$, and $q'=q/(q-1)$ is the conjugate exponent. (I will keep employing the letters $t$, $q$ and $r$ in these mentioned meanings: $t$ for type, $q$ for cotype, and $r$ for Fou\emph{r}ie\emph{r}-type. The letter $p$, which is frequently used for type, is here reserved for a generic exponent of the Lebesgue space $L^p$.)

Note that the new result indeed improves the old one, as Fourier-type $r$ implies both type $t\geq r$ and cotype $q\leq r'$, but neither implication is reversible: e.g. $L^p$, for $p\in(1,\infty)$, has type $t=\min(2,p)$ and cotype $q=\max(2,p)$, hence $\max(t,q')=2$, but only Fourier-type $r=\min(p,p')$. In fact, having $\max(t,q')=2$ is quite typical for the ``common'' UMD spaces appearing in analysis, and the present result shows that H\"ormander's multiplier theorem holds for functions valued in such spaces in its exact classical form. This is not so exciting in the usual $L^p$ spaces, given that the same result in this case could be derived (as is well known) directly from H\"ormander's original result with an application of Fubini's theorem. However, the above mentioned type and cotype properties remain valid also in the noncommutative $L^p$ spaces \cite{Fac}; hence the same is true for the multiplier theorem, and in this case the conclusion appears to be new and nontrivial. (For an interesting recent application of the multiplier theory in noncommutative $L^p$ spaces, although not of the particular result obtained here, see \cite{PotSuk09}.)

The key novelty of the proof is applying an improvement of the contraction principle in the presence of cotype due to myself and Veraar, which was already used by us for relaxing Fourier-type assumptions in certain other results~\cite{HV:smooth}. Unfortunately, this proof does not extend to operator-valued multipliers, so that I cannot fully claim the fall of the Fourier type in the vector-valued multiplier theory.

While the result described above seems to set the new record for scalar-valued multipliers, I have no reason to propose that it should be the best possible. The index $\max(t,q')$ seems more likely to be an artificial product of the proof than an eternal truth. In fact, there appears to be no known reason why the classical H\"ormander theorem, for scalar-valued multipliers, could not be valid in all UMD spaces without any extra conditions.

The reader may recall that Girardi and Weis \cite[Remark~4.5]{GW} do make an assertion concerning the sharpness of the order $\floor{n/r}+1$ by referring to an example in the stability theory of semigroups, which goes back to \cite[Remark 3.7]{Weis:stability} and \cite[Sec.~4]{WW}. But, first of all, this example is set up in the space $X=L^p\cap L^{p'}$ for which $t=q'=r=\min(p,p')$, so the multiplier conditions involving type and cotype or Fourier-type are indistinguishable. Second, the example is an operator-valued one, namely, the square of the resolvent of the infinitesimal generator. Thus neither the possibility of the order $\floor{n/\max(t,q')}+1$ being sufficient even in the operator-valued case, nor the order $\floor{n/2}+1$ being sufficient in the scalar-multiplier case, is directly excluded by the mentioned example. However, proving or disproving either of these conjectures remains a challenge for further investigation.

\subsection*{Acknowledgements}
I would like to thank Mark Veraar who insisted that there should be an application of our paper \cite{HV:smooth} to the theory of Fourier multipliers. I was funded by the Academy of Finland through the projects ``Vector-valued singular integrals'' and ``$L^p$ methods in harmonic analysis''.

\section{Different Mihlin conditions}

Following the approach in \cite{Hyt:Mihlin}, rather than the classical Mihlin multiplier condition, it will be more convenient to consider the ``Mihlin--H\"older condition''
\begin{equation}\label{eq:MihGamma}
  \Babs{\prod_{i:\alpha_i=1}(I-\tau_{h_i e_i})m(\xi)}\lesssim\abs{h^{\alpha\gamma}}\abs{\xi}^{-\abs{\alpha}\gamma},\qquad
  \abs{\xi}>2\abs{h},
\end{equation}
where $\tau_h m(\xi):=m(\xi-h)$ is the translation, $\gamma\in(0,1]$ is a fixed parameter, and the multi-index notation $h^{\alpha}:=\prod_{i=1}^n h_i^{\alpha_i}$ is employed. The condition \eqref{eq:MihGamma} has the advantage of being a simultaneous generalization of both the Mihlin and H\"ormander type assumptions, as shown in the following. More general considerations of kind are found in \cite{Hyt:Mihlin}, but a short argument is provided here for completeness.

\begin{lemma}\label{lem:HolderVsHorm}
Suppose that $\abs{\partial^{\alpha}m(\xi)}\lesssim\abs{\xi}^{-\abs{\alpha}}$ for all $\alpha\in\{0,1\}^n$ such that $\abs{\alpha}\leq\ceil{n\gamma}:=\min\{k\in\Z:k\geq\ceil{n\gamma}\}$. Then \eqref{eq:MihGamma} holds for all $\alpha\in\{0,1\}^n$.
\end{lemma}

\begin{proof}
Note that $(I-\tau_{h_i e_i})m(\xi)=h_i\int_0^1\partial_i m(\xi-t_i h_i e_i)\ud t_i$, and iterating this in all the relevant variables gives
\begin{equation*}
  \prod_{i:\alpha_i=1}(I-\tau_{h_i e_i})m(\xi)
  =h^{\alpha}\int_{[0,1]^{\abs{\alpha}}}\partial^{\alpha}m\big(\xi-\sum_{i:\alpha_i=1}t_i h_i e_i\big)
   \prod_{i:\alpha_i=1}\ud t_i.
\end{equation*}
For $\abs{\alpha}\leq\ceil{n\gamma}$, the integrand may be directly estimated by $\abs{\xi-\sum t_i h_i e_i}^{-\abs{\alpha}}\lesssim\abs{\xi}^{-\abs{\alpha}}$, and the claim follows since $\abs{h^{\alpha}}\abs{\xi}^{-\abs{\alpha}} \lesssim(\abs{h^{\alpha}}\abs{\xi}^{-\abs{\alpha}})^{\gamma} =\abs{h^{\alpha\gamma}}\abs{\xi}^{-\abs{\alpha\gamma}}$.

For $\abs{\alpha}>\ceil{n\gamma}$, one first notes that the expression on the left of \eqref{eq:MihGamma} is dominated by the sum of $2^{\abs{\alpha-\beta}}$ similar expressions with the multi-index $\beta\leq\alpha$ in place of $\alpha$. Choosing $\abs{\beta}=\ceil{n\gamma}$, the earlier considerations give, for the left side of \eqref{eq:MihGamma}, the upper bound $\abs{h^{\beta}}\abs{\xi}^{-\abs{\beta}}$. Take the geometric average of these upper bounds over all the $\binom{\abs{\alpha}}{\ceil{n\gamma}}$ choices of $\beta\leq\alpha$ with $\abs{\beta}=\floor{n\gamma}$ to obtain the new upper bound $\abs{h^{\alpha\ceil{n\gamma}/\abs{\alpha}}}\abs{\xi}^{-\ceil{n\gamma}}$. As $\ceil{n\gamma}/\abs{\alpha}\geq n\gamma/n=\gamma$, the estimate is complete.
\end{proof}

\section{The main result}

The main result, already sketched in the introduction, is the following:

\begin{theorem}\label{thm:MihCotype}
Let $X$ be a UMD space with type $t$ and cotype $q$. Then the Mihlin--H\"older condition \eqref{eq:MihGamma} for all $\alpha\in\{0,1\}^n$ and some $\gamma>1/\max(t,q')$ is sufficient for the boundedness of $T_m$ on $L^p(\R^n;X)$ for all $p\in(1,\infty)$. In particular, the Mihlin--H\"ormander condition $\abs{\partial^{\alpha}m(\xi)}\lesssim\abs{\xi}^{-\abs{\alpha}}$ for all $\alpha\in\{0,1\}^n$ with $\abs{\alpha}\leq\floor{n/\max(t,q')}+1$ is sufficient.
\end{theorem}

\begin{proof}
The heart of the matter consists of proving that the Mihlin--H\"older condition of order $\gamma>1/q'$ is sufficient for the $L^p(\R^n;X)$ boundedness of $T_m$ provided that $X$ is a UMD space and $L^p(\R^n;X)$ (rather than just $X$) has cotype $q$. So let us assume this (returning to the original assumptions at the end of the proof), and choose some auxiliary number $\tilde{q}$ so that $\gamma>1/\tilde{q}'>1/q'$ (thus in particular $\tilde{q}\in(q,\infty)$).

The proof is modelled after the approach in \cite{Hyt:Mihlin}, with an application of a result from \cite{HV:smooth} at one critical point.
Choose a usual Littlewood--Paley function $\phi_0\in\mathscr{S}(\R^n)$ so that $\sum_{j\in\Z}\hat\phi_0(2^j\xi)=1$ for all $\xi\neq 0$, as well as $0\leq\hat\phi_0\leq\hat\chi_0:=1_{\{2^{-1}\leq\abs{\xi}\leq 2\}}$, and let $\phi_j(x):=2^{jn}\phi_0(2^j x)$, $\chi_j(x):=2^{jn}\chi_0(2^j x)$. I will use Bourgain's vector-valued Littlewood--Paley inequality \cite[Theorem~3]{Bourgain}, which involves the  independent random signs $\eps_j$ on some probability space $\Omega$ with $\prob(\eps_j=1)=\prob(\eps_j=-1)=\frac{1}{2}$. I write simply $\Norm{\ }{p}$ for the $L^p$ norm, whether on $L^p(\R^n;X)$ or on $L^p(\R^n\times\Omega;X)$, which should be understood from the context in such a way that all the variables are always ``integrated out''. Then $T_m f=\check{m}*f=K*f$ is estimated by
\begin{equation*}
\begin{split}
  \Norm{K*f}{p}
  &\eqsim\BNorm{\sum_j\eps_j(\phi_j*K)*(\chi_j*f)}{p}
  =\BNorm{\sum_j\eps_j\int_{\R^n}(\phi_j*K)(y)\tau_y(\chi_j*f)\ud y}{p} \\
  &\leq\int_{\R^n}\BNorm{\sum_j\eps_j 2^{-jn}(\phi_j*K)(2^{-j}y)\tau_{2^{-j}y}(\chi_j*f)}{p}\ud y \\
  &\leq\int_{\R^n}\BNorm{\sum_j\eps_j 2^{-jn}(\phi_j*K)(2^{-j}y)(\chi_j*f)}{p}\log(2+\abs{y})\ud y.
\end{split}
\end{equation*}
by the Littlewood--Paley inequality, the triangle inequality after a change of variables, and Bourgain's estimate for the translation operators \cite[Lemma 10]{Bourgain}.

For $\alpha\in\{0,1\}^n$, $\mu\in\N^{\alpha}:=\{\nu\in\N^n:\nu_i=0\text{ if }\alpha_i=0\}$, define (as in \cite[Sec.~3]{Hyt:Mihlin})
\begin{equation*}
\begin{split}
  E(\alpha) &:=\{x\in\R^n:\abs{x_i}\leq 1\text{ if }\alpha_i=0;\abs{x_i}>1\text{ if }\alpha_i=1\},\\
  E(\alpha,\mu) &:=\{x\in E(\alpha): 2^{\mu_i}<\abs{x_i}\leq 2^{\mu_i+1}\text{ if }\alpha_i=1\}.
\end{split}
\end{equation*}
so that disjointly
\begin{equation*}
  \R^n=\bigcup_{\alpha\in\{0,1\}^n}E(\alpha),\quad
  E(\alpha)=\bigcup_{\mu\in\N^{\alpha}}E(\alpha,\mu).
\end{equation*}

Then, with $f_j:=\chi_j*f$, $K_j(y):=2^{-jn}(\phi_j*K)(2^{-j}y)$, one estimates
\begin{equation}\label{eq:useHolder}
\begin{split}
  &\int_{E(\alpha,\mu)}\BNorm{\sum_j\eps_j K_j(y)f_j}{p}\log(2+\abs{y})\ud y \\
  &\lesssim\Big(\int_{E(\alpha,\mu)}\BNorm{\sum_j\eps_j K_j(y)f_j}{p}^{\tilde q}\ud y\Big)^{1/\tilde{q}}
      2^{\abs{\mu}/\tilde{q}'}(1+\abs{\mu}) \\
\end{split}
\end{equation}
by H\"older's inequality.

Now comes the key step. Recall that $L^p(\R^n;X)$ has cotype $q<\tilde{q}$. This is exactly the condition of \cite[Lemma~3.1(2)]{HV:smooth} under which an $L^{\tilde{q}}$-version of the contraction principle holds, allowing to pull out the multiplying functions $K_j(y)$ from the randomized norm as follows:
\begin{equation}\label{eq:useHV}
  \lesssim\sup_j\Big(\int_{E(\alpha,\mu)}\abs{K_j(y)}^{\tilde q}\ud y\Big)^{1/\tilde{q}}
      \BNorm{\sum_j\eps_j f_j}{p} 2^{\abs{\mu}/\tilde{q}'}(1+\abs{\mu}).
\end{equation}
The Littlewood--Paley inequality says that
\begin{equation*}
   \BNorm{\sum_j\eps_j f_j}{p}\eqsim\Norm{f}{p}.
\end{equation*}

For $y\in E(\alpha,\mu)$, the functions $1-\exp(i2\pi y_k 2^{-\mu_k-3})$, for those $k$ with $\alpha_k=1$, are bounded away from zero, and hence the integral involving $K_j$ can be estimated by
\begin{equation}\label{eq:useHY}
\begin{split}
  &\Big(\int_{E(\alpha,\mu)}\abs{K_j(y)}^{\tilde q}\ud y\Big)^{1/\tilde{q}} \\
  &\lesssim\Big(\int_{E(\alpha,\mu)}\Babs{\prod_{k:\alpha_k=1}
      [1-\exp(i2\pi y_k 2^{-\mu_k-3})]K_j(y)}^{\tilde{q}}\ud y\Big)^{1/\tilde{q}} \\
  &=\Big(\int_{E(\alpha,\mu)}\Babs{\mathscr{F}^{-1}\Big\{\prod_{k:\alpha_k=1}
      [I-\tau_{e_k 2^{-\mu_k-3}}]\hat{K}_j\Big\}(y)}^{\tilde q}\ud y\Big)^{1/\tilde{q}} \\
  &\lesssim\Big(\int_{\R^n}\Babs{\prod_{k:\alpha_k=1}
      [I-\tau_{e_k 2^{-\mu_k-3}}][\hat{\phi}_0(\cdot)m(2^j\cdot)](\xi)}^{\tilde{q}'}\ud\xi\Big)^{1/\tilde{q}'}
   \lesssim 2^{-\abs{\mu}\gamma},
\end{split}
\end{equation}
by the Hausdorff--Young inequality (for a scalar-valued function, thus avoiding any reference to Fourier-type!) in the second-to-last step, and the assumed Mihlin--H\"older conditions \eqref{eq:MihGamma} in the last one: these are invariant under the replacement of $m(\xi)$ by $m(2^j\xi)$, and in combination with the regularity and support properties of $\hat\phi_0$, they give the stated bound.

Thus it has been shown that
\begin{equation*}
\begin{split}
  \Norm{K*f}{p}
  &\lesssim\sum_{\alpha\in\{0,1\}^n}\sum_{\mu\in\N^{\alpha}}\int_{E(\alpha,\mu)}\BNorm{\sum_j\eps_jK_j(y)f_j}{p}\log(2+\abs{y})\ud y \\
  &\lesssim\sum_{\alpha\in\{0,1\}^n}\sum_{\mu\in\N^{\alpha}}2^{-\abs{\mu}\gamma}\Norm{f}{p}2^{\abs{\mu}/\tilde{q}'}(1+\abs{\mu})
   \lesssim\Norm{f}{p},
\end{split}
\end{equation*}
since the series converges by the choice that $1/\tilde{q}'<\gamma$. This completes the proof of the boundedness of $T_m$ when $m$ satisfies the Mihlin--H\"older condition \eqref{eq:MihGamma} with $\gamma>1/q'$ and $L^p(\R^n;X)$ has cotype $q$.

If it is only assumed, as in the statement of the theorem, that $X$ has cotype $q$, it is still true that $L^p(\R^n;X)$ has cotype $q$ when $p\in(1,q]$. Thus the previous conclusion holds in this range. But the multiplier condition implies that the associated kernel satisfies H\"ormander's integral condition (cf.\ \cite[Sec.~7]{Hyt:Mihlin})
\begin{equation*}
  \int_{\abs{x}>2\abs{y}}\abs{K(x-y)-K(x)}\ud x\lesssim 1,
\end{equation*}
so that the $L^p(\R^n;X)$ boundedness for just one $p\in(1,\infty)$ already bootstraps to all $p\in(1,\infty)$ by the classical vector-valued extension of the Calder\'on--Zygmund theory due to Benedek, Calder\'on and Panzone \cite{BCP}, and the additional condition on the cotype of $L^p(\R^n;X)$ has been removed.

Let then $X$ have type $t$, which should allow for a Mihlin--H\"older condition of order $\gamma>1/t$. But the type $t$ property of $X$ implies (and in the case of a UMD space, is equivalent to) the cotype $t'$ for $X^*$, which is also a UMD space. By what has already been proven, the multiplier theorem holds in $L^{p'}(\R^n;X^*)$ for Mihlin--H\"older multipliers of order $\gamma>1/t''=1/t$. But the adjoint of $T_m$ with respect to the standard duality of $L^p(\R^n;X)$ and $L^{p'}(\R^n;X^*)$ is $T_{\tilde{m}}$, where $\tilde{m}(\xi)=m(-\xi)$, and this reflection preserves the class of Mihlin--H\"older multipliers. Hence the multiplier theorem with $\gamma>1/t$, and thus with $\gamma>\min(1/t,1/q')=1/\max(t,q')$, is also valid in $L^p(\R^n;X)$.

Finally, the assertion of Theorem~\ref{thm:MihCotype} concerning Mihlin--H\"ormander multipliers is an immediate consequence of the Mihlin--H\"older version by Lemma~\ref{lem:HolderVsHorm}, and hence the theorem is now completely proven.
\end{proof}

\section{Variations and observations}

An inspection of the proof leads to the following corollary, in which the sufficiency of the classical H\"ormander condition is achieved in all UMD spaces by restricting to special classes of multipliers. The case of multipliers supported in a dyadic annulus is due to McConnell \cite{McConnell} by a different method. Until now, the only available proof was McConnell's original one, so that this result remained somewhat isolated from the more recent developments of the vector-valued multiplier theory (such as \cite{GW,Hyt:Mihlin}) building on the work of Bourgain \cite{Bourgain}.

\begin{corollary}
Let $X$ be a UMD with type $t$ and cotype $q$, and let $m:\R^n\setminus\{0\}\to\C$ be either supported in a dyadic annulus $r<\abs{\xi}<2r$, or positively homogeneous of order zero, i.e., $m(\lambda\xi)=m(\xi)$ for $\lambda>0$. Then the Mihlin--H\"older condition \eqref{eq:MihGamma} for all $\alpha\in\{0,1\}^n$ and some $\gamma>1/2$ is sufficient for the boundedness of $T_m$ on $L^p(\R^n;X)$ for all $p\in(1,\infty)$. In particular, the Mihlin--H\"ormander condition $\abs{\partial^{\alpha}m(\xi)}\lesssim\abs{\xi}^{-\abs{\alpha}}$ for all $\alpha\in\{0,1\}^n$ with $\abs{\alpha}\leq\floor{n/2}+1$ is sufficient.
\end{corollary}

\begin{proof}
If $m$ is positively homogeneous (even the dyadic version $m(2^j\xi)=m(\xi)$ actually suffices), then $K_j(y)\equiv K_0(y)$, and the estimate \eqref{eq:useHV} where cotype $q$ of $L^p(\R^n;X)$ was used to pull out the supremum of the $L^{\tilde{q}}$ norms of the $K_j$ becomes trivial without any assumptions. The same applies if $m$ is supported in a dyadic annulus, in which case only three different $K_j$ are nonzero. Since \eqref{eq:useHV} was the only place where the cotype assumption played a role in the proof of Theorem~\ref{thm:MihCotype}, we can repeat the argument for the special multipliers of either kind in an arbitrary UMD space $X$. Note, however, that one still needs to use the scalar-valued Hausdorff--Young inequality with exponent $\tilde{q}$, which gives the restriction $\tilde{q}\geq 2$.
\end{proof}

\begin{remark}
It may be interesting to shortly recall the proof of the operator-multiplier version of Theorem~\ref{thm:MihCotype} (due to Girardi and Weis \cite{GW}, and fine-tuned in \cite{Hyt:Mihlin}, whose approach is taken here). In the assumptions of the mentioned theorem, $\max(t,q')$ is replaced by a Fourier type $r$ of $X$, and the condition on the multiplier involves an $R$-bounded version (see \cite{GW} for this notion) of the Mihlin--H\"older estimate, for all $\alpha\in\{0,1\}^n$ and some $\gamma>1/r$,
\begin{equation}\label{eq:MihGammaR}
  \rbound\big(\{\abs{\xi}^{\abs{\alpha}\gamma}\abs{h^{-\alpha\gamma}}
    \prod_{i:\alpha_i=1}[I-\tau_{h_i e_i}]m(\xi): \abs{\xi}>2\abs{h}\}\big)\lesssim 1.
\end{equation}
By an argument similar to Lemma~\ref{lem:HolderVsHorm}, this is implied by the $R$-bounded Mihlin--H\"ormander condition $\rbound(\{\abs{\xi}^{\abs{\alpha}}\partial^{\alpha}m(\xi):\xi\in\R^n\setminus\{0\}\})\lesssim 1$ for all $\alpha\in\{0,1\}^n$ with $\abs{\alpha}\leq\floor{n/r}+1$.

As for the proof of the multiplier theorem, again, it may be assumed that even $L^p(\R^n\times\Omega;X)$ has Fourier-type $r$, as this is the case for $p\in[r,r']$, and for other values of $p$ one extrapolates the boundedness of $T_m$ by the Benedek--Calder\'on--Panzone theory \cite{BCP}. Now the estimates \eqref{eq:useHolder} through \eqref{eq:useHY} are replaced by
\begin{equation*}
\begin{split}
  &\int_{E(\alpha,\mu)}\BNorm{\sum_j\eps_j K_j(y)f_j}{p}\log(2+\abs{y})\ud y \\
  &\lesssim\Big(\int_{E(\alpha,\mu)}\BNorm{\sum_j\eps_j K_j(y)f_j}{p}^{r'}\ud y\Big)^{1/r'}
      2^{\abs{\mu}/r}(1+\abs{\mu}) \\
  &\lesssim\Big(\int_{\R^n}\BNorm{\prod_{k:\alpha_k=1}
      [1-\exp(i2\pi y_k 2^{-\mu_k-3})]\sum_j\eps_j K_j(y)f_j}{p}^{r'}\ud y\Big)^{1/r'} 2^{\abs{\mu}/r}(1+\abs{\mu}) \\
  &\lesssim\Big(\int_{\R^n}\BNorm{\sum_j\eps_j\prod_{k:\alpha_k=1}
      [I-\tau_{e_k 2^{-\mu_k-3}}][\hat{\phi}_0(\cdot)m(2^j\cdot)](\xi)f_j}{p}^{r}\ud\xi\Big)^{1/r} 2^{\abs{\mu}/r}(1+\abs{\mu}) \\
  &\lesssim \BNorm{\sum_j\eps_j f_j}{p}2^{-\abs{\mu}\gamma}2^{\abs{\mu}/r}(1+\abs{\mu})
    \lesssim \Norm{f}{p}2^{-\abs{\mu}(\gamma-1/r)}(1+\abs{\mu}),
\end{split}
\end{equation*}
where the Hausdorff--Young inequality for an $L^p(\R^n\times\Omega;X)$-valued function was used (under the assumption that the mentioned space has Fourier-type $r$) in the third step, and the $R$-bounded Mihlin condition \eqref{eq:MihGammaR} in the fourth one.

The key novelty of the proof of Theorem~\ref{thm:MihCotype} was pulling out (with the help of \cite[Lemma 3.1(2)]{HV:smooth})  the functions $K_j$ before the application of the Hausdorff--Young inequality, so that it could be applied to a scalar-valued function instead of a vector-valued one. In the operator-multiplier case, this trick is unavailable for two reasons: first, we do not have a version of \cite[Lemma 3.1(2)]{HV:smooth} to pull out operator-valued functions, and second, pulling them out would probably only make things worse as the operator space $\bddlin(X)$, unlike the scalar field $\C$, has even weaker properties than the space $X$.
\end{remark}

\bibliography{multiplier-literature}
\bibliographystyle{plain}

\end{document}